\documentclass[11pt, reqno]{article}
\title{\bf On the tail of the branching random walk local time}
\author{{\bf Omer Angel, Tom Hutchcroft, and Antal J\'arai}}

\usepackage{amsmath,amssymb,amsfonts,amsthm}
\usepackage{color}

\usepackage[colorlinks=true,linkcolor=blue,citecolor=blue,urlcolor=blue,pdfborder={0 0 0}]{hyperref}
\usepackage{cleveref}
\usepackage{caption}
\usepackage{thmtools}
\usepackage{mathrsfs}
\usepackage{extarrows}
\usepackage{titling}
\usepackage{commath}
\usepackage{mathtools}
\usepackage{titlesec}
\usepackage{bbm}
\usepackage{enumitem}
\usepackage{cite}

\usepackage[margin=1.07in]{geometry}
\usepackage{setspace}
\crefname{theorem}{Theorem}{Theorems}
\crefname{thm}{Theorem}{Theorems}
\crefname{lemma}{Lemma}{Lemmas}
\crefname{lem}{Lemma}{Lemmas}
\crefname{remark}{Remark}{Remarks}
\crefname{prop}{Proposition}{Propositions}
\crefname{defn}{Definition}{Definitions}
\crefname{corollary}{Corollary}{Corollaries}
\crefname{conjecture}{Conjecture}{Conjectures}
\crefname{question}{Question}{Questions}
\crefname{chapter}{Chapter}{Chapters}
\crefname{section}{Section}{Sections}
\crefname{figure}{Figure}{Figures}

\theoremstyle{plain}
\newtheorem{thm}{Theorem}[section]

\newtheorem{lemma}[thm]{Lemma}

\newtheorem{prop}[thm]{Proposition}

\theoremstyle{definition}

\theoremstyle{remark}
\newtheorem{remark}[thm]{Remark}

\numberwithin{equation}{section}

\renewcommand{\P}{\mathbb P}
\newcommand{\E}{\mathbb E}

\newcommand{\R}{\mathbb R}
\newcommand{\Z}{\mathbb Z}

\newcommand{\bv}{\mathbf v}
\newcommand{\bx}{\mathbf x}
\newcommand{\bD}{\mathbf{D}}
\newcommand{\bG}{\mathbf{G}}

\newcommand{\eps}{\varepsilon}

\newcommand{\cH}{\mathcal H}

\newcommand{\cS}{\mathcal{S}}
\newcommand{\cT}{\mathcal T}

\pretitle{\begin{flushleft}\Large}
\posttitle{\par\end{flushleft}\vskip 0.5em}
\preauthor{\begin{flushleft}}
\postauthor{
\\
\vspace{0.5em}
\footnotesize{Department of Mathematics, University of British Columbia. Email: 
\href{mailto:angel@math.ubc.ca}{angel@math.ubc.ca}}
\\
\footnotesize{Statslab, DPMMS, University of Cambridge. Email: 
\href{mailto:t.hutchcroft@maths.cam.ac.uk}{t.hutchcroft@maths.cam.ac.uk}}
\\
\footnotesize{Probability Laboratory, University of Bath. Email: 
\href{mailto:a.jarai@bath.ac.uk}{a.jarai@bath.ac.uk}}

\end{flushleft}}
\predate{\begin{flushleft}}
\postdate{\par\end{flushleft}}

\renewenvironment{abstract}
 {\par\noindent\textbf{\abstractname.}\ \ignorespaces}
 {\par\medskip}

\titleformat{\section}
{\normalfont\large \bf}{\thesection}{0.4em}{}

\begin{document}

\maketitle

\begin{abstract}
  Consider a critical branching random walk on $\Z^d$, $d\geq 1$, started with a single particle at the origin, and let $L(x)$ be the total number of particles that ever visit a vertex $x$.
  We study the tail of $L(x)$ under suitable conditions on the offspring distribution.
  In particular, our results hold if the offspring distribution has an exponential moment.
\end{abstract}

\section{Introduction}

In this paper we study the tail of the number of times a critical branching random walk on $\Z^d$ returns to the origin.  The result is most interesting in the \emph{upper-critical dimension} $d=4$, where we find that the local time has a stretched-exponential tail.

\begin{thm} 
\label{thm:main}
Let $d \geq 1$, let $( B_n )_{n\geq0}$ be a branching random walk on $\Z^d$ whose offspring distribution $\mu$ is critical, non-trivial, and sub-exponential, started with a single particle at the origin, and let $L(0)$ be the total number of particles that visit the origin. Then
\[
\P_{\mu,0}\left(L(0) \geq n\right) =
\begin{cases}
 \Theta\bigl(n^{-2/(4-d)}\bigr) & \qquad d<4\\
\exp\left[ -\Theta (\sqrt{n})\right] & \qquad d=4\\
\exp\left[ -\Theta(n)\right] & \qquad d > 4
\end{cases}
\]
for every $n\geq 1$.
\end{thm}

Here, we say that the offspring distribution $\mu$ is \textbf{critical} if it has mean $1$, \textbf{non-trivial} if $\mu(1)<1$, and \textbf{sub-exponential} if there exist positive constants $C$ and $c$ such that $\mu(n) \leq C e^{-cn}$ for every $n\geq 1$. We use both ``$f(n)=\Theta(g(n))$ for every $n\geq 1$'' and ``$f(n) \asymp g(n)$ for every $n\geq 1$'' to mean that there exist positive constants $c,C$ depending only on the offspring distribution $\mu$ and the dimension $d$ such that
$cg(n) \leq f(n) \leq C g(n)$ for every $n\geq 1$. Similar meaning applies to the symbols $\preceq$ and $\succeq$, so that, for example, ``$f_n(x) \preceq g_n(x)$ for every $n\geq 1$ and $x\in \Z^d$'' means that there exists a positive  constant $C$ depending only on the offspring distribution $\mu$ and the dimension $d$ such that $f_n(x) \leq C g_n(x)$ for every $n \geq 1$ and $x \in \Z^d$.

Our work is motivated in part by our hope to understand the analogous questions for the \emph{Abelian sandpile model} \cite{Dhar90,MajDhar92}. These questions remain open even in the high-dimensional case, in which other aspects of the model are now fairly well-understood \cite{JarRed08,bhupatiraju2016inequalities,hutchcroft2018universality}.

There is an extensive literature on critical branching random walk on $\Z^d$, with works particularly relevant to the present paper including \cite{BeCu12,MR3395472,MR3480967,zhu2017critical,MR3962482,zhu2016criticalI,zhu2016criticalII,MR2678898,MR2778804}. 
In light of this extensive literature, we were surprised to find that the tail of the local time had not previously been studied. 
The basic methods that we use (inductive analysis of moments via diagrammatic sums) are well-known to experts, but we have included a detailed exposition so that this paper could be used as an introduction to these techniques.


We also prove the following off-diagonal version of \cref{thm:main}. 
We use the notation \mbox{$\langle x \rangle = 2 \vee d(0,x)$}, where $d(0,x)$ denotes the graph distance between $0$ and $x$, to avoid dividing by zero.

\begin{thm} 
\label{thm:main_offdiagonal}
Let $d \geq 1$, let $(B_n)_{n\geq0}$ be a branching random walk on $\Z^d$ whose offspring distribution is critical, non-trivial, and sub-exponential, started with a single particle at the origin, and let $L(x)$ be the total number of particles that visit $x$. Then
\[
\P_{\mu,0}(L(x) \geq n) \asymp
\begin{cases}
\min\left\{n^{-2/(4-d)}, \langle x \rangle^{-2} \right\}
  & \qquad d<4\\[10pt]
\exp\bigg[ -\Theta \!\left(\min\left\{\sqrt{n}, \frac{n}{\log \langle x \rangle}\right\}\right)\bigg] \langle x \rangle^{-2} \log^{-1}\langle x \rangle & \qquad d=4\\[10pt]
\exp\Big[ -\Theta(n)\Big] \langle x \rangle^{-d+2} & \qquad d > 4
\end{cases}
\]
for every $n \geq 1$ and $x \in \Z^d$.
\end{thm}

The proof of \cref{thm:main_offdiagonal} relies on the asymptotics of the hitting probability 
\begin{equation}
\label{eq:intro_hitting}
\P_{\mu,0}\!\left(L(x) \geq 1\right) \asymp \begin{cases}
\langle x \rangle^{-2} & d \leq 3\\
\langle x \rangle^{-2} \log^{-1} \langle x \rangle & d =4\\
\langle x \rangle^{-d+2} & d \geq 5
\end{cases} \qquad \text{ for every $x \in \Z^d$,}
\end{equation}
which were computed by Le Gall and Lin \cite{MR3395472,MR3480967} for $d \neq 4$ and by Zhu \cite{MR3962482,zhu2017critical} for $d=4$.

\begin{remark}
It is well-known  that a critical branching random walk \emph{conditioned to survive forever} visits the origin infinitely often if and only if $d \leq 4$ \cite{BeCu12}. This is closely related to the fact that the conditional distribution of $L(x)$ given $L(x)>0$ is tight as $x \to \infty$ if and only if $d \geq 5$. 
\end{remark}

\begin{remark}
In the context of super-Brownian motion (which is a continuum analogue of critical branching random walk), Le Gall and Merle \cite{MR2266716} studied the conditional distribution of the occupation measure $\mathcal{Z}(B_1(x))$ of the unit ball $B_1(x)$ for large $x$, given that this measure is positive. Their results are closely related to \cref{thm:main_offdiagonal}. 
In particular, they show that if $d=4$ then  the conditional distribution of the normalized occupation measure $\mathcal{Z}(B_1(x))/\log |x|$ given that it is positive converges to an exponential distribution as $x\to \infty$. 
It would be interesting to establish a version of their theorem in the discrete case.
\end{remark}

\begin{remark}
  It is natural to consider the distribution of $L(x)$ for branching random walks on graphs other than $\Z^d$.
It should be straightforward to adapt the proof of \cref{thm:main} to bounded degree graphs that are $d$-Ahlfors regular and satisfy Gaussian heat kernel estimates. See e.g.\ \cite{KumagaiBook,Woess} for background on these notions. We restrict attention to the usual nearest-neighbour random walk on $\Z^d$ for clarity of exposition.
\end{remark}




\section{Background}
\subsection{Branching random walk}

Let us now very briefly define the model, referring the reader to e.g.\ \cite{LP:book,zhu2016criticalI} for more details on branching processes and Galton-Watson trees.
Given $d\geq 1$, an offspring distribution $\mu$ (i.e., a probability measure on $\{0,1,\ldots\}$), and a point $x \in \Z^d$ we write $\P_{\mu,x}$ for the law of a branching random walk $(B_n)_{n\geq 0}$ on $\Z^d$ with offspring distribution $\mu$ started with a single particle at $x$. 
More precisely, $(B_n)_{n\geq 0}$ is a Markov chain whose state space is the set of finitely supported functions $\Z^d \to \{0,1\ldots\}$, where $B_0(y) = \mathbbm{1}(y=x)$ and where we think of $B_n(y)$ as the number of particles occupying the point $y$ at generation $n$. At each time step, each particle splits into a random number of offspring particles independently at random according to the offspring distribution $\mu$, and each offspring particle immediately performs an independent simple random walk step.
We define the \textbf{local time} $L_n(x) = \sum_{m =0}^n B_m(x)$ to be the total number of particles that occupy the site $x$ up to time $n$, and similarly define the limit $L(x) = \sum_{m=0}^\infty B_m(x)$.

Alternatively, we may construct branching random walk by first taking a Galton-Watson tree $T$ with offspring distribution $\mu$, which encodes the genealogy of the particles of the branching random walk,  letting $X:V(T) \to \Z^d$ be a uniform random graph homomorphism from $T$ into $\Z^d$ mapping the root to $x$ (i.e., a simple random walk on $\Z^d$ started at $x$ and indexed by $T$), and letting $B_n(y) = \#\{v \in \partial T_n : X(v) =y \}$ for every $n\geq 0$ and $y\in \Z^d$. We write $\partial T_r$ for the set of vertices of $T$ at distance exactly $r$ from the root. It is easily seen that if $\mu$ is critical then 
$\E_{\mu,0}[\# \partial T_r]= 1$ for every $r\geq 0$. Moreover, if $\mu$ is critical, non-trivial, and has finite variance $\sigma^2$, then \emph{Kolmogorov's estimate} states that
\begin{equation}
\P_{\mu,0}(\partial T_r \neq \emptyset) \sim \frac{2}{\sigma^2r} \qquad \text{ as $r \to \infty$}.
\end{equation}
This estimate was proven by Kolmogorov under a third moment assumption \cite{KolmogorovBiology}, and in full generality by Kesten, Ney, and Spitzer \cite{MR0207052}; see \cite{LPP95} for a modern proof.

\subsection{Random walk estimates}

We now briefly recall the relevant background concerning random walk on $\mathbb{Z}^d$, referring the reader to e.g. \cite{Woess,KumagaiBook} for further background. Let $p_n(u,v)$ denote the $n$-step transition probabilities of simple random walk on $\mathbb{Z}^d$. The \textbf{Gaussian heat kernel estimates} state that 
\begin{equation}
\label{eq:GHKE}
 p_n(x,y) +  p_{n+1}(x,y) \asymp n^{-d/2}\exp\left[-\Theta\Bigl( d(x,y)^2/n \Bigr) \right]
\end{equation}
for every $x,y \in \Z^d$ and $n \geq 1$, where $d(x,y)$ denotes the graph distance between $x$ and $y$.
(Note that the constants in the $\Theta$ notation may differ for the lower and upper bounds.)
Note that $p_n(x,y)=0$ if $n$ has a different parity to $d(x,y)$.
In particular, we have that
\begin{equation}
\label{eq:local_limit}
 p_{2n}(x,x) \asymp n^{-d/2}
\end{equation}
for every $x\in \Z^d$ and $n\geq 1$. 
If $d \geq 3$, the Gaussian heat kernel estimates can be integrated over time to obtain that the Green's function $\bG(u,v)=\sum_{n\geq0} p_n(u,v)$ satisfies
\begin{equation}
\label{eq:Greens}
 \mathbf{G}(u,v) \asymp  d(u,v)^{-d+2}
\end{equation} 
for every $u,v \in \Z^d$.


\section{Diagrammatic expansion of moments}


In this section we discuss how the moments of the branching random walk local time may be expanded in terms of \emph{diagrammatic sums}, and then prove a recursive inequality that may be used to bound these sums. This basic methodology is well-known to experts, see e.g.\ \cite[eq. 6]{MR2266716} for an application to super-Brownian motion, and \cite{MR762034} for related techniques in percolation.

Recall that a \textbf{rooted plane tree} is a locally finite tree with a distinguished root vertex and a distinguished linear ordering of the children of each vertex; an isomorphism of trees is an isomorphism of rooted plane trees if it preserves this additional data. Note that a rooted plane tree cannot have any nontrivial automorphisms.  We may consider a Galton-Watson tree $T$ to be a rooted plane tree by picking a uniform random linear ordering of the children of every vertex.

Let $k\geq 0$. We define a \textbf{$k$-labelled rooted plane tree} to be a finite rooted plane tree $S$ with vertex set $V(S)$, together with a (not necessarily injective) labelling function $\ell:\{0,1,\ldots,k\} \to V(S)$ mapping $0$ to the root of $S$ such that every leaf of $S$ is labelled (i.e., is in the image of $\ell$).
Note that leaves of $S$ may have multiple labels, and that internal vertices may also have labels.
Given a $k$-labelled rooted plane tree $S$, we write $\partial V (S) = \ell(\{0,1,\ldots,k\})$ and $V^\circ(S)=V(S) \setminus \partial V(S)$ to denote the sets of labelled and unlabelled vertices of $S$. An isomorphism of rooted plane trees is an isomorphism of labelled rooted plane trees if it preserves the labelling.

We say that a $k$-labelled rooted plane tree is a \textbf{(labelled) $k$-skeleton} if every unlabelled vertex has at least two children.
In particular, up to isomorphism there is only one $0$-skeleton, which has one vertex labelled $0$ and no edges. Similarly, there are exactly two isomorphism classes of $1$-skeletons, which have one and two vertices respectively. 
For each $k\geq 0$, we let $\mathcal{S}_k$ be a set of isomorphism class representatives for the set of labelled $k$-skeletons and let $\cH_k$ be a set of isomorphism class respresentatives for the set of $k$-labelled rooted plane trees.


We will use the modified Green's function
\begin{align*}
  \tilde \bG(x,y) &= \sum_{k\geq 1} p_k(x,y) = \bG(x,y)-\mathbbm{1}(x=y), \\
  \intertext{and}
  \tilde \bG_n(x,y) &= \sum_{k= 1}^n p_k(x,y) = \bG_n(x,y)-\mathbbm{1}(x=y)
\end{align*}
for each $x,y \in \Z^d$ and $n\geq 1$.

For each $k\geq 0$, each $k$-labelled rooted plane tree $S$, and each $\mathbf{x} = (x_0,\ldots,x_k) \in (\Z^d)^{k+1}$ we write $\Lambda(\mathbf{x};S)=\Lambda(x_0,\ldots,x_k;S)$ for the set $\mathbf{y} = (y_u)_{u \in V(S)} \in (\Z^d)^{V(S)}$ such that $y_{\ell(i)}=x_i$ for every $0\leq i \leq k$.
(This set is empty if $\ell(i)=\ell(j)$ but $x_i\neq x_j$.)
When $S$ is a $k$-skeleton, we define the \textbf{$S$-diagram} to be the function $\mathbf{D}( \;\cdot\; ; S) : (\Z^d)^{k+1} \to [0,\infty]$ given by
\[
\mathbf{D}(\mathbf{x};S) = \bD(x_0,\ldots,x_k;S) =  \sum_{\mathbf{y} \in \Lambda(\mathbf{x};S)} \prod_{u \sim v} \tilde \bG(y_u,y_v),
\]
where the second product is over all unordered pairs of adjacent vertices in $S$.
In particular, if $S$ is the $0$-skeleton then $\mathbf{D}( \;\cdot\; ; S) \equiv 1$, while if $S$ is the $1$-skeleton with two vertices then $\mathbf{D}( x,y ; S) \equiv \tilde \bG(x,y)$.
Similarly, for each $k,n\geq 0$ and each $k$-skeleton $S$ we define the \textbf{truncated $S$-diagram} to be the function $\mathbf{D}_n( \;\cdot\; ; S) : (\Z^d)^{k+1} \to [0,\infty]$ given by
\[
\mathbf{D}_n(\mathbf{x};S) = \bD_n(x_0,\ldots,x_k;S) = \sum_{\mathbf{y} \in \Lambda(\mathbf{x};S)} \prod_{u \sim v} \tilde \bG_n(y_u,y_v),
\]
where, as before, the second product is over all unordered pairs of adjacent vertices in $S$.


Recall that  $\E_{\mu,x}$ denotes the law of a branching random walk $(B_n)_{n\geq 0}$ with offspring distribution $\mu$ started with a single particle at $x$. Recall also that we write $L_n(y) = \sum_{k=0}^n B_k(y)$ for the total number of particles that visit $y$ up to time $n$, and write $L(y)=\sum_{k=0}^\infty B_k(y)$ for the total number of particles that ever visit $y$. 
For each $k\geq 0$, we define $b_k$ to be the expectation of the binomial coefficient $\binom{\cdot}{k}$
 under the offspring distribution $\mu$, that is,
\[
b_k = \sum_{n= k}^\infty \binom{n}{k}\mu(n).
\]
In particular, $b_0=b_1=1$ when $\mu$ is critical,  $b_k<\infty$ if and only if $\mu$ has a finite $k$th moment, and $b_k>0$ if and only if $\sum_{n \geq k} \mu(n)>0$. In particular $b_2>0$ whenever $\mu$ is critical and non-trivial. 
For each vertex $v$ in a rooted plane tree $S$, we write $c(v)$ for the number of children of $v$.

\begin{prop}[Diagrammatic expansion of moments] 
\label{lem:moments_to_diagrams}
Let $\mu$ be critical and let $d\geq 1$. We have that
\begin{align*}\E_{\mu,x_0}\left[\prod_{i=1}^k L(x_i) \right]  &= \sum_{S \in \mathcal{S}_k} \mathbf{D}(x_0,x_1,\ldots,x_k;S)  \prod_{v \in V(S)} b_{c(u)}\\
\intertext{and}
\E_{\mu,x_0}\left[\prod_{i=1}^k L_n(x_i) \right] &\leq \sum_{S \in \mathcal{S}_k}  \mathbf{D}_n(x_0,x_1,\ldots,x_k;S) \prod_{v \in V(S)} b_{c(u)}
\end{align*} for every $n,k\geq 0$ and $x_0,\ldots,x_k \in \Z^d$. 
\end{prop}

\begin{proof}
We first explain the appearance of the combinatorial term $\prod b_{c(u)}$ in the proposition.
Let $T$ be the genealogical tree of $B$, and let $X$ be the random embedding of $T$ into $\Z^d$.
Let $H$ be a $k$-labelled rooted plane tree. We say that a graph homomorphism $\phi$ from $H$ into the Galton-Watson tree $T$ is an \textbf{embedding} if it is injective, maps the root of $H$ to the root of $T$, and respects the plane structure of $H$ and $T$ in the sense that for every vertex $v$ of $H$ with children $u_1,\ldots,u_n$, the children $\phi(u_1),\ldots,\phi(u_n)$ of $\phi(v)$ in $T$ appear in the same linear order as $u_1,\ldots,u_n$ do in $H$. However, $T$ may have additional vertices not corresponding to any vertex in $H$. It is easily seen by induction on the height of $H$ that
\begin{equation}
\label{eq:Hembeddings}
\E_{\mu,x_0}\left[ \#\!\left\{\text{embeddings of $H$ into $T$}\right\} \right] = \prod_{u \in H} b_{c(u)}
\end{equation}
for every finite rooted plane tree $H$. (This equality holds even if $\mu$ is not critical.)

\medskip

We begin with the first, non-trucated formula.
Given a $k$-tuple of not necessarily distinct vertices $\bv = (v_1,\dots,v_k) \in V(T)^k$, let $H(\bv)$ be the $k$-labelled rooted plane tree spanned by the union of the geodesics between the root of $T$ and the vertices $v_1,\ldots,v_k$, with labelling function defined by setting $\ell(0)$ to be the root of $T$ and setting $\ell(i)=v_i$ for each $1 \leq i \leq k$.
We can write
\begin{align*}
\prod_{i=1}^k L(x_i) &= \#\Bigl\{\bv \in V(T)^k : X(v_i)=x_i \; \forall 1 \leq i \leq k\Bigr\}.
\\&= \sum_{H \in \cH_k} 
\#\Bigl\{\bv \in V(T)^k :  H(\bv) \cong H,\; X(v_i)=x_i \; \forall 1 \leq i \leq k\Bigr\}.
\end{align*}
 On the other hand, by definition of the embedding $X$ we have that
\begin{align*}
\E\left[\prod_{i=1}^k L(x_i) \bigg| T \right] 
&=
\sum_{H \in \cH_k}\#\Bigl\{(\bv \in V(T)^k :  H(\bv) \cong H \Bigr\} \sum_{\mathbf{y} \in \Lambda(\mathbf{x},H)}\prod_{u \sim v} p_1(y_u,y_v).
\\
&= \sum_{H \in \cH_k} \#\{\text{embeddings of $H$ into $T$}\} \sum_{\mathbf{y} \in \Lambda(\mathbf{x},H)} \prod_{u \sim v} p_1(y_u,y_v),
\end{align*}
where $p_1(\cdot,\cdot)$ denotes the one-step transition probabilities for simple random walk on $\Z^d$.
Taking expectations over $T$ and applying \eqref{eq:Hembeddings}, we obtain that
\begin{equation}
\label{eq:Hexpansion}
\E_{\mu,x_0}\left[\prod_{i=1}^k L(x_i)  \right] = \sum_{H \in \cH_k} \sum_{y \in \Lambda(x,H)}\prod_{u \in H} b_{c(u)} \prod_{v = c(u)} p_1(y_u,y_v).
\end{equation}

For each $H$ in $\cH_k$, let $S(H) \in \cS_k$ denote the $k$-skeleton obtained from $H$ by replacing each path whose interior vertices are unlabelled vertices of degree two by a single edge. Thus, for each $S\in \cS_k$, the set of $H \in \cH_k$ with $S(H) = S$ is equal to the set of $k$-labelled rooted plane trees that can be obtained from $S$ by replacing each edge with a path of arbitrary length. Since $\mu$ is critical and $b_1=1$, one may readily verify that
\[
\sum_{\substack{H \in \cH_k\\S(H)=S}} \sum_{y \in \Lambda(x,H)}\prod_{u \in H} b_{c(u)} \prod_{v = c(u)} p_1(y_u,y_v) = \sum_{y \in \Lambda(x,S)}\prod_{u \in V(S)} b_{c(u)} \prod_{v = c(u)} \tilde \bG(y_u,y_v)
\]
for every $S \in \cS_k$ and $x_0,x_1,\ldots,x_k \in \Z^d$.
The first claim follows from this together with \eqref{eq:Hexpansion}.

The proof in the truncated case is fairly similar, and we give only a very brief outline.
For each $n \geq 0$ and $k\geq 0$, let $\cH_{n,k} \subset \cH_k$ denote the set of $k$-labelled rooted plane trees with height at most $n$, and let $\cH_{n,k}'$ denote the set of $k$-labelled rooted plane trees in which each path whose interior vertices are unlabelled vertices of degree two has length at most $n$. Clearly $\cH_{n,k} \subset \cH_{n,k}'$. We have by similar reasoning to above that
\begin{align*}
\E\left[\prod_{i=1}^k L_n(x_i)\right] &= \sum_{H \in \cH_{n,k}} \sum_{y \in \Lambda(x,H)} \prod_{u \in H} b_{c(u)} \prod_{v=c(u)}p_1(y_u,y_v)\\
&\leq \sum_{H \in \cH_{n,k}'} \sum_{y \in \Lambda(x,H)} \prod_{u \in H} b_{c(u)} \prod_{v=c(u)}p_1(y_u,y_v)\\
&= \sum_{S \in \cS_k}  \sum_{y \in \Lambda(x,S)} \prod_{u \in V(S)} b_{c(u)} \prod_{v=c(u)}\tilde \bG_n (y_u,y_v)
\end{align*}
as claimed.
\end{proof}


We next state and prove a recursive inequality that allows us to bound the diagrammatic sums arising in \cref{lem:moments_to_diagrams}. 
For each $k\geq 0$, let $\cS_k'$ be the set of $k$-skeletons whose labelling function is injective.
We observe that for any tuple $\bx$, the maximum $\max_{S\in\cS'_k} \bD(\bx;S)$ is invariant to permuting the elements of $\bx$.
Indeed, $\bD(\bx;S)$ is invariant under applying the same permutation to both the entries of $\bx$ and the labels of $S$.
(If $0$ is not a fixed point of the permutation, this requires one to change the root of $S$.)
The symmetry of the random walk implies that such re-rooting also does not change $\bD$.
In light of this, for each $k\geq 1$ and $x\in \Z^d$, we define
\[
M_k(x) : = \max_{S \in \cS_k'} \bD(0,\ldots,0,x;S) = \max_{S \in \cS_k'} \bD(x,0,\ldots,0;S) = \max_{S \in \cS_k'} \bD(0,x,\ldots,x;S),
\]
where the equality of these three expressions follow from the symmetry noted above.
We could equivalently define $M_k(x)$ by maximizing $\bD(\bx;S)$ over all $S \in \cS_k'$ and all $\bx$ which are a permutation of $(0,\dots,0,x)$.
Similarly, we define the truncated version
\[
M_{k,n}(x) : = \max_{S \in \cS_k'} \bD_n(0,\ldots,0,x;S) = \max_{S \in \cS_k'} \bD_n(x,0,\ldots,0;S) = \max_{S \in \cS_k'} \bD_n(0,x,\ldots,x;S)
\]
for each $k\geq 0$ and $n\geq 0$.
Note that $M_1(x)=\tilde \bG(0,x)$ and $M_{1,n}(x)=\tilde \bG_n(0,x)$ for every $x\in \Z^d$ and $n\geq 0$.

\begin{lemma}[Recursive inequality for the maximal diagram]
  \label{lem:recursion}
  Let $d\geq 1$ and $k\geq 2$. Then
  \begin{equation}
    M_k(x) \leq  \left[1 \vee  \tilde\bG(0,0)^{-1}\right]
    \max_{0<r<k} \left\{ \sum_{y \in \Z^d} M_{r}(y) M_{k-r}(y) \tilde \bG(y,x) \right\}
    \label{eq:recursive_inequality}
  \end{equation}
  and
  \begin{equation}
    M_{k,n}(x) \leq \left[1 \vee  \tilde\bG_n(0,0)^{-1}\right]
    \max_{0<r<k}\left\{ \sum_{y \in \Z^d} M_{r,n}(y) M_{k-r,n}(y) \tilde \bG_n(y,x) \right\}.
    \label{eq:recursive_inequality_truncated}
  \end{equation}
\end{lemma}

\medskip
Note that the quantities $1\vee\tilde \bG_n(0,0)^{-1}$ and $1 \vee \tilde \bG(0,0)^{-1}$ are bounded above by $p_2(0,0)^{-1}=2d$ when $n\geq 2$. Be warned, however that $1\vee\tilde \bG_n(0,0)^{-1}$ is infinite when $n \in \{0,1\}$. Later in the paper we will be careful to avoid this case.

\begin{proof}[Proof of \cref{lem:recursion}]
We will prove \eqref{eq:recursive_inequality}, the proof of \eqref{eq:recursive_inequality_truncated} being almost identical.
It suffices to prove that
  \begin{equation}
  M_k(x) \leq \max_{0<r<k}\Bigl\{ M_{r}(x)M_{k-r}(x) \Bigr\}
  \,\vee\, M_{k-1}(0) \tilde \bG(0,x) \, \vee\, \max_{0<r<k}\left\{ \sum_{y \in \Z^d} M_{r}(y) M_{k-r}(y) \tilde \bG(y,x) \right\}
\label{eq:recursion_detail}
\end{equation}
for every $k\geq 2$. Indeed, the first and second terms are each clearly smaller than the third multiplied by $M_1(0)^{-1}=\tilde \bG(0,0)^{-1}$ (consider the contributions to the sum in the third term from $y=0$ and $y=x$).

Let $k\geq 2$, let $S\in \cS_k'$, let $x\in \Z^d$, and let $\mathbf{x} =(0,\ldots,0,x) \in (\Z^d)^{k+1}$.  
We consider three cases, which correspond to the three terms being maximized over in the inequality \eqref{eq:recursion_detail}:
\begin{enumerate}
  \item $\ell(k)$ is not a leaf.
  \item $\ell(k)$ is a leaf and the parent of $\ell(k)$ is in $\partial V(S)$ (i.e. is labelled).
  \item $\ell(k)$ is a leaf and the parent of $\ell(k)$ is in $V^\circ(S)$ (i.e., is unlabelled).
\end{enumerate}

\paragraph{Case 1.} Let $a \geq 1$ be the number of labelled vertices that are descendants of $\ell(k)$ in $S$.  Since $\ell$ is injective, $\ell(k)$ is not the root of $S$ and $a<k$.
 Let $S_1$ be the $a$-skeleton formed by $\ell(k)$ and its descendants in $S$, where we consider $\ell(k)$ to be the root of $S_1$ and re-index the  labels if necessary so that the labelling function has domain $\{0,\ldots,a\}$. Similarly, let $S_2=(T_2,\ell_2)$ be the $(k-a)$-skeleton obtained from $S$ by deleting all the descendants of $\ell(k)$, and re-indexing the labels so that the labelling function $\ell_2$ has domain $\{0,1,\ldots,k-a\}$ and satisfies $\ell_2(k-a)=\ell(k)$. (In both cases, the details of relabelling are not important.) Having done this, we observe that, by the definitions, 
 \begin{equation}
\bD(0,\ldots,0,x;S) = \bD(x,0,\ldots,0;S_1)\bD(0,\ldots,0,x;S_2) \leq M_a(x)M_{k-a}(x).
\end{equation}
We deduce that if $S \in \cS_k'$ is such that $\ell(k)$ is not a leaf of $S$ then
 \begin{equation}
 \label{eq:case1}
\bD(0,\ldots,0,x;S)  \leq \max\Bigl\{ M_r(x)M_{k-r}(x) : 1 \leq r \leq k-1 \Bigr\},
\end{equation}
which corresponds to the first term in \eqref{eq:recursion_detail}.

\paragraph{Case 2.}
 We may define a $(k-1)$-skeleton $S'$ by deleting $\ell(k)$ from $S$.
   The definitions then ensure that
 \begin{equation}
 \label{eq:case2}
\bD(0,\ldots,0,x;S) = \mathbf{D}(0,\ldots,0;S') \tilde \bG(0,x) \leq  M_{k-1}(0)\tilde \bG(0,x),
 \end{equation}
 which corresponds to the second term in \eqref{eq:recursion_detail}.

\paragraph{Case 3.}
Let $v$ be the (unlabelled) parent of $\ell(k)$. Let $a$ be the number of labelled descendants of $v$ other than $\ell(k)$. Since $v$ is unlabelled it has at least two children, and therefore has $a \geq 1$.
Let $S_1$ be the $a$-skeleton consisting of $v$ and its descendants other than $\ell(k)$, where we consider $v$ to be the root of $S_1$ and re-index the other labels as appropriate. Similarly, let $S_2$ be the $(k-a)$-skeleton obtained from $S$ by deleting all the descendants of $v$ (but not $v$ itself), re-indexing all the remaining labelled vertices to have labels in $\{0,\ldots,k-a-1\}$, and giving $v$ the label $k-a$. (The details of how this is done are not important.) It follows from the definitions that
\[
\bD(0,\ldots,0,x;S)=\sum_{y\in \Z^d} \bD(0,\ldots,0,y;S_2) \bD(y,0,\ldots,0;S_1) \tilde \bG(y,x) \leq \sum_{y\in \Z^d} M_a(y) M_{k-a}(y) \tilde \bG(y,x).
\]
We deduce that if $S\in \cS_k'$ is such that $\ell(k)$ is a leaf and the parent of $\ell(k)$ is in $V^\circ(S)$ then
\begin{equation}
\label{eq:case3}
\bD(0,\ldots,0,x;S) \leq \max_{0<r<k} \left\{ \sum_{y\in \Z^d} M_r(y) M_{k-r}(y) \tilde \bG(y,x) \right\},
\end{equation}
which corresponds to the third term in \eqref{eq:recursion_detail}.

\medskip 

Since one of the three cases above holds for every $S \in \cS_k'$, the claimed inequality \eqref{eq:recursion_detail} follows from \eqref{eq:case1}, \eqref{eq:case2}, and \eqref{eq:case3}.
\end{proof}

We now note that bounds on $M_k$ and $M_{k,n}$ yield bounds on \emph{all} diagrams, i.e. also with non-injective labels.
Indeed, suppose that $S \in \cS_k$ for some $k\geq 1$ and that the labelling function of $S$ is not injective. Let $r=|\ell(\{0,\ldots,k\})|-1$, let $\sigma:\{0,\ldots,r\} \to \{0,\ldots,k\}$ be defined recursively by $\sigma(0)=0$ and $\sigma(i) = \min\{ j > \sigma(i-1) : \ell(j) \notin \ell(\{0,\ldots,\sigma(i-1)\})\}$ for each $1 \leq i \leq r$, and let $S'$ be the $r$-skeleton with the same underlying rooted plane tree as $S$ and with labelling function $\ell'(i)=\ell(\sigma(i))$. Then it follows from the definitions that
\begin{multline*}
\bD(x_0,x_1,\ldots,x_k;S) =\\ \mathbbm{1}\Bigl(x_i=x_j \text{ for every $0\leq,i,j \leq k$ with $\ell(i)=\ell(j)$}\Bigr) \bD(x_0,x_{\sigma(1)},\ldots,x_{\sigma(r)};S')
\end{multline*}
for every $x_0,x_1,\ldots,x_k \in \Z^d$. In particular, it follows that
\begin{equation}
\label{eq:max_noninjective}
\max_{S \in \cS_k} \bD(0,\ldots,0,x;S) \leq \max_{0 \leq r \leq k} M_r(x)
\end{equation}
for every $k\geq 0$ and $x\in \Z^d$. Similar reasoning gives that
\begin{equation}
\label{eq:max_noninjective_truncated}
\max_{S \in \cS_k} \bD_n(0,\ldots,0,x;S) \leq \max_{0 \leq r \leq k} M_{r,n}(x)
\end{equation}
for every $k\geq 0$, $x\in \Z^d$, and $n\geq 0$.

\section{Low dimensions}

In this section we prove the following proposition, which implies the case $d<4$ of \cref{thm:main,thm:main_offdiagonal}.
We remark that in this low dimensional case we do not require a sub-exponential tail for the offspring distribution, and a moment condition is sufficient.

\begin{prop} 
  \label{prop:low_dim}
  Suppose either that $d\in\{1,2\}$ and that the offspring distribution $\mu$ is critical, non-trivial, and has finite second moment, or that $d=3$ and the offspring distribution $\mu$ is critical, non-trivial, and has finite third moment. Then
  \[
    \P_{\mu,0}(L(x) \geq n) \asymp \min\left\{n^{-2/(4-d)}, \langle x \rangle^{-2} \right\}
    \qquad \text{ for every $n\geq 1$ and $x\in \Z^d$.}
  \]
  for every $n\geq 1$.
\end{prop}

\begin{remark}
  One can also obtain from our proof that if $d=3$ and $\mu$ has finite second moment then \[n^{-2} \log^{-1} (n+1) \preceq \P_{\mu,0}(L(0) \geq n) \preceq n^{-2} \log (n+1)\]
  for every $n\geq 1$.
\end{remark}


Our analysis is informed by the following heuristic: In low dimensions, the easiest way for the local time $L(x)$ to be large is for the genealogical tree to be sufficiently large, without any other unusual behaviour for the tree or the associated random walks.
Indeed, intuitively, if the genealogical tree survives to generation $k$, which occurs with probability $\Theta(k^{-1})$, then it typically contains roughly $k^2$ vertices, and the locations of the corresponding particles are roughly uniformly distributed on the ball of radius $k^{1/2}$. Thus, if $R$ denotes the survival time of the branching random walk, we should typically have that $L(x) =0$ if $R \ll \langle x \rangle^2$ and that $L(x)$ is $\Theta(R^{(4-d)/2})$ if $R = \Omega(\langle x \rangle^2)$.
Thus, we expect that the easiest way to have $L(x) \geq n$ is for $R$ to be at least $\min \left\{ \langle x \rangle^2, n^{2/(4-d)}\right\}$, which leads to the expression given in \cref{prop:low_dim}.
One may think of this heuristic argument as yielding a \emph{hyperscaling relation} for branching random walk below the critical dimension, and the proof of \cref{prop:low_dim} as a rigorous verification of this hyperscaling relation.

\medskip

We now begin the rigorous proof of \cref{prop:low_dim}. 
We shall see that it is sufficient to look at the first three moments of the truncated local time $L_n(x)$. (In dimensions $d=1,2$ it suffices to consider the first and second moment, while in $d=3$ dimensions using the second moment results in an unwanted logarithmic correction.)

\pagebreak

\begin{lemma} 
\label{lem:low_d_moments}
Let $\mu$ be critical, and let $d\in\{1,2,3\}$. Then the following moment bounds hold.
\begin{enumerate}[nosep,label=(\alph*)]
\item If $\mu$ has finite second moment then 
\begin{equation}
\label{eq:low_d_2nd_moment}
\E_{\mu,0}\left[L_n(x)^2\right] \preceq \begin{cases} n^{3-d} &  d \leq 2\\ 
  \log (n+1) & d=3
  \end{cases} 
  \qquad \text{ for every $x\in \Z^d$ and $n\geq 1$.} \end{equation} 
\item If $\mu$ has finite third moment, then \begin{equation}
\label{eq:low_d_3rd_moment}
\E_{\mu,0}\left[L_n(x)^3\right] \preceq  n^{(10-3d)/2} 
\qquad \text{ for every $x\in \Z^d$ and $n\geq 1$.}
\end{equation}
\end{enumerate}
\end{lemma}

Note that these bounds are clearly not sharp when, say, $\langle x \rangle \gg \sqrt{n}$. This will not be a problem for us as the estimates are sharp in the regimes that we wish to apply them.

We will frequently use the easily proved fact that for every $c >0$ and $\alpha \in \R$ there exists a constant $C=C(c,\alpha)$ such that
\[
\sum_{r \geq 1} r^{\alpha} \exp\left[-c r^2/n \right] \leq C \begin{cases}
n^{(1+\alpha)/2} & \alpha>-1\\
\log (n+1) & \alpha = -1\\
1 & \alpha <-1
\end{cases}
\qquad \text{for every $n\geq 1$.}
\]

\begin{proof}[Proof of \cref{lem:low_d_moments}]
It suffices to consider the case $n \geq 2$, so that $1 \vee \tilde \bG_n(0,0)^{-1} \leq 4d^2 \preceq 1$.

\vspace{-5mm}
\paragraph{(a) Second moment.}
Let $S \in \cS_2$ be a $2$-skeleton. 
Fix $1 \leq d \leq 3$.
No $2$-skeleton has a vertex of degree more than three.
Since $b_0,b_1,b_2 < \infty$ by assumption, and there is a finite number (10) of $2$-skeletons, it suffices by \cref{lem:moments_to_diagrams} to prove that 
\begin{equation}
M_{k,n}(x) \preceq \begin{cases} n^{3-d} &  d \leq 2\\ 
  \log (n+1) & d=3\end{cases} \qquad \text{ for every $x\in \Z^d$ and $n\geq 1$.} 
  \label{eq:M2bound}
\end{equation}
for every $k=0,1,2$ and $x\in \Z^d$. This bound is trivially satisfied for $k=0$, since in this case $M_{0,n}(x) = \mathbbm{1}(x=0) \leq 1$. For $k=1$ we have that
\begin{equation}
\label{eq:kis1}
M_{1,n}(x) = \tilde \bG_n(0,x) \preceq \sum_{k=1}^n k^{-d/2} \preceq \begin{cases} n^{1/2} & d=1 \\
\log(n+1) & d=2\\
1 & d =3,
\end{cases}
\end{equation}
which is of lower order than the required bound.
For $k=2$, we apply \cref{lem:recursion} to deduce that
\begin{equation}
\label{eq:M2_expanded}
M_{2,n}(x) \preceq \sum_{y \in \Z^d} 
\tilde \bG_n(0,y)^2 \tilde \bG_n(y,x)
\end{equation}
for every $x \in \Z^d$ and $n\geq 2$.
By H\"older's inequality this yields
\[
  M_{2,n}(x) \preceq
  \bigg(\sum_{y \in \Z^d} \tilde \bG_n(0,y)^3 \bigg)^{2/3}
  \bigg(\sum_{y \in \Z^d} \tilde \bG_n(y,x)^3 \bigg)^{1/3}
  = \sum_{y \in \Z^d} \tilde \bG_n(0,y)^3.
\]

Applying the Gaussian heat kernel estimates \cref{eq:GHKE} we deduce that there exists a positive constant $c$ such that
\begin{equation}
M_{2,n}(x)
 \\\preceq 
\sum_{y \in \Z^d} \sum_{k_1=1}^n \sum_{k_2=1}^n \sum_{k_3=1}^n
k_1^{-d/2}k_2^{-d/2}k_3^{-d/2}\exp\left[ -\frac{c \langle y \rangle^2}{k_1} -\frac{c \langle y \rangle^2}{k_2}-\frac{c \langle x-y \rangle^2}{k_3} \right].
\end{equation}
Using that $\#\{y \in \Z^d: \langle y \rangle=r\} = O(r^{d-1})$ and changing variables to $z=x-y$ if $k_3=\min \{k_1,k_2,k_3\}$, we have that
\begin{align*}
\sum_{y \in \Z^d} \exp\left[ -\frac{c \langle y \rangle^2}{k_1} -\frac{c \langle y \rangle^2}{k_2}-\frac{c \langle x-y \rangle^2}{k_3} \right] &\preceq \sum_{r=1}  r^{d-1} \exp\left[-\frac{cr^2}{\min\{k_1,k_2,k_3\}}\right]\\ &\preceq \min\{k_1,k_2,k_3\}^{d/2}
\end{align*}
and hence that
\[
M_{2,n}(x) \preceq \sum_{k_1=1}^n \sum_{k_2=1}^n \sum_{k_3=1}^n k_1^{-d/2}k_2^{-d/2}k_3^{-d/2} \min \{k_1,k_2,k_3\}^{d/2}.
\]
If $d \in \{1,2\}$, we bound $\min\{k_1,k_2,k_3\}^{d/2} \leq k_1^{d/6} k_2^{d/6} k_3^{d/6}$ and deduce that
\[
M_{2,n}(x) \preceq \sum_{k_1=1}^n \sum_{k_2=1}^n \sum_{k_3=1}^n k_1^{-d/3}k_2^{-d/3}k_3^{-d/3}
\preceq
n^{3-d}
\]
for every $n\geq 2$ as claimed. Meanwhile, if $d=3$, we compute that
\[
M_{2,n}(x) \preceq  \sum_{k_1=k_3}^n \sum_{k_2=k_3}^n \sum_{k_3=1}^n k_1^{-3/2}k_2^{-3/2} \preceq \sum_{k_3=1}^n k_3^{-1} \preceq \log (n+1)
\]
for every $n\geq 2$ as claimed.

\paragraph{(b) Third moment.} 
Since no $3$-skeleton has a vertex with more than 3 offspring,
and since $b_0,b_1,b_2,b_3 < \infty$ by assumption, 
it suffices by \cref{lem:moments_to_diagrams} to prove that 
\[
M_{k,n}(x) \preceq n^{(10-3d)/2}
\]
for every $n\geq 2$, $k=0,1,2,3$, and $x\in \Z^d$. The fact that this bound is satisfied for $k=0,1,2$ has already been established.
For $k=3$, we apply \cref{lem:recursion} and \eqref{eq:M2_expanded} to deduce that
\[
M_{3,n}(x) \preceq \sum_{y \in \Z^d} \sum_{z\in \Z^d} \tilde \bG_n(0,z)^2 \tilde \bG_n(z,y) \tilde \bG_n(0,y) \tilde \bG_n(y,x).
\]
As before, we apply the Gaussian heat kernel estimates \eqref{eq:GHKE} to bound this sum by
\begin{multline*}
M_{3,n}(x)
 \preceq 
\sum_{y,z \in \Z^d} \sum_{1 \leq k_1,\ldots,k_5 \leq n}
k_1^{-d/2}k_2^{-d/2}k_3^{-d/2}k_4^{-d/2}k_5^{-d/2}\\
\cdot\exp\left[ -\frac{c \langle z \rangle^2}{k_1} -\frac{c \langle z \rangle^2}{k_2} - \frac{c \langle z-y \rangle^2}{k_3}  -\frac{c \langle y \rangle^2}{k_4} -\frac{c \langle x-y \rangle^2}{k_5} \right].
\end{multline*}
By similar reasoning to above, we can bound 
\begin{multline*}
\sum_{y,z \in \Z^d} \exp\left[ -\frac{c \langle z \rangle^2}{k_1} -\frac{c \langle z \rangle^2}{k_2} - \frac{c \langle z-y \rangle^2}{k_3}  -\frac{c \langle y \rangle^2}{k_4} -\frac{c \langle x-y \rangle^2}{k_5} \right] 
\\\preceq \sum_{z \in \Z^d} \exp\left[ -\frac{c \langle z \rangle^2}{k_1} -\frac{c \langle z \rangle^2}{k_2}\right] \min\{k_3,k_4,k_5\}^{d/2}\\
\preceq 
\min\{k_1,k_2\}^{d/2} \min \{k_3,k_4,k_5\}^{d/2}.
\end{multline*}
By symmetry we can also bound the left hand side by 
$\min\{k_1,k_2,k_3\}^{d/2} \min \{k_4,k_5\}^{d/2}$.
Now observe that, using that $\min \{k_3,k_4,k_5\} \leq k_3^{2/5} k_4^{3/10} k_5^{3/10}$
\begin{multline*}
\min\left\{\min\{k_1,k_2\}^{d/2} \min \{k_3,k_4,k_5\}^{d/2}, \min\{k_1,k_2,k_3\}^{d/2} \min \{k_4,k_5\}^{d/2} \right\}
\\\leq \min\left\{k_1^{d/4}k_2^{d/4} k_3^{d/5}k_4^{3d/20}k_5^{3d/20}, k_1^{3d/20}k_2^{3d/20}k_3^{d/5}k_4^{d/4}k_5^{d/4} \right\}
\leq \prod_{i=1}^5 k_i^{d/5},
\end{multline*}
where we once again bounded the minimum by the geometric mean  and used that $(1/4+3/20)/2 = 2/5$ in the final inequality.
Thus, we may bound
\[M_{3,n}(x)
 \preceq 
 \sum_{1 \leq k_1,\ldots,k_5 \leq n} \prod_{i=1}^5 k_i^{-3d/10} \preceq \left(n^{1-3d/10}\right)^5 = n^{(10-3d)/2}
 \]
 for every $n\geq 1$ and $x \in \Z^d$ as required.
\end{proof}

Before applying \cref{lem:low_d_moments} to prove \cref{prop:low_dim}, let us recall the Paley-Zygmund inequality and its higher-moment variants. The usual Paley-Zygmund inequality states that if $X$ is a non-negative random variable with finite second moment then
\[
\P\left(X \geq \eps \E[X]\right)\geq \frac{(1-\eps)^2 \E[X]^2}{\E[X^2]}
\]
for every $0\leq \eps \leq 1$. Applying this inequality to the conditional distribution of a non-negative random variable $X$ given that $X>0$ and doing a little algebra, we obtain that in fact
\[
\P\left(X \geq \eps \E\left[X \mid X>0\right]\right)\geq \frac{(1-\eps)^2 \E[X]^2}{\E[X^2]}
\]
for every $0 \leq \eps \leq 1$. 

The Paley-Zygmund inequlity also has the following $L^p$ version. We include a short proof since this inequality is less standard.

\begin{lemma} Let $X$ be a non-negative random variable. Then
\[
\P\left(X \geq \eps \E\left[X\right]\right) \geq \frac{(1-\eps)^{p/(p-1)} \E[X]^{p/(p-1)}}{\E[X^p]^{1/(p-1)}}
\]
for every $p >1$ and $0 \leq \eps \leq 1$.
\end{lemma}

\begin{proof}
H\"older's inequality implies that
\begin{align*}
\E[X] &\leq \eps \E\left[X \right]\P\left(X < \eps \E\left[X \right] \right) + \E\left[X \mathbbm{1}\left(X \geq \eps \E\left[X \right]\right)\right]\\
&\leq \eps \E\left[X\right]+ \E\left[X^p\right]^{1/p} \P\left(X \geq \eps \E\left[X\right]\right)^{(p-1)/p}.
\end{align*}
Rearranging gives the desired inequality.
\end{proof}

Now suppose that $X$ is a nonnegative random variable. Applying the above inequality to a random variable $Z$ distributed according to the conditional distribution of $X$ given $X>0$ gives that
\begin{multline}
\P\left(X \geq \eps \E\left[X \mid X>0\right]\right) = \P(X>0)\P\left(Z \geq \eps \E[Z]\right) \\
\geq 
\P(X>0)\frac{(1-\eps)^{p/(p-1)} \E[Z]^{p/(p-1)}}{\E[Z^p]^{1/(p-1)}} = 
\frac{(1-\eps)^{p/(p-1)} \E[X]^{p/(p-1)}}{\E[X^p]^{1/(p-1)}}
\label{eq:generalized_Paley_Zygmund}
\end{multline}
for every $p>1$ and $0\leq \eps \leq 1$.

\begin{proof}[Proof of \cref{prop:low_dim}]
Let $1 \leq d \leq 3$. We assume that $\mu$ has finite second moment if $d\in \{1,2\}$ and that $\mu$ has finite third moment if $d =3$. 
We begin with the upper bounds.
We have by \eqref{eq:intro_hitting} that
\[
\P(L(x) \geq n ) \leq \P(L(x)\geq 1) \leq \E L(x) \asymp \langle x \rangle^{-2}
\]
for every $n\geq 1$ and $x\in \Z^d$.
Thus it suffices to prove that
\[
\P(L(x) \geq n ) \preceq n^{-2/(4-d)}
\]
for every $n\geq 1$ and $x\in \Z^d$.
Since $\mu$ has finite second moment, is critical and non-trivial, we have that
\[
\P_{\mu,0}(\partial T_r \neq \emptyset)\asymp \frac{1}{r} \qquad \text{ for every $r\geq 1$},
\]
where $T$ is the genealogical tree of the branching process. Thus, we can bound 
\begin{multline*}
\P_{\mu,0}(L(x) \geq n) \leq 
\P_{\mu,0} \left(L_r(x)\geq n \right) + \P_{\mu,0}(\partial T_r \neq \emptyset)
\\
\leq
\frac{1}{n^d} \E_{\mu,0} \left[L_r(x)^d\right] + \P_{\mu,0}(\partial T_r \neq \emptyset)
\preceq 
\begin{cases}
n^{-1} r^{1/2} + r^{-1} & d=1\\
n^{-2} r + r^{-1} & d=2\\
n^{-3} r^{1/2} + r^{-1} & d=3
\end{cases}
\end{multline*}
for every $n,r\geq 1$. Taking $r= n^{2/3}$ when $d=1$, $r=n$ when $d=2$, and $r=n^2$ when $d=3$, we obtain that
\[
\P_{\mu,0}(L(x) \geq n) \preceq n^{-2/(4-d)}
\]
for every $n\geq 1$ and $x\in \Z^d$ as desired.

\medskip

We now turn to the lower bounds. It suffices to prove that there exists a constant $c$ such that
\[
\P_{\mu,0}(L(x) \geq n) \succeq n^{-2/(4-d)} 
\]
for every $x\in \Z^d$ and every $n \geq c \langle x \rangle^{4-d}$: the required bound for smaller $n$ follows since $\P_{\mu,0}(L(x) \geq n)$ is a decreasing function of $n$. 
For each $r \geq 1$ we have  by linearity of expectation that
\begin{align*}
\E_{\mu,0}\left[\sum_{\ell=r}^{2r} B_\ell (x) \bigg|\sum_{\ell=r}^{2r} B_\ell (x) >0 \right] &\geq 
\E_{\mu,0}\left[\sum_{\ell=r}^{2r} B_\ell (x) \bigg| \partial T_r \neq \emptyset\right] \\
&= \E_{\mu,0}\left[\sum_{\ell=r}^{2r} |\partial T_\ell|  p_\ell(0,x) \right] \P(\partial T_r \neq \emptyset)^{-1} \asymp r \sum_{\ell=r}^{2r}p_\ell(0,x)
\end{align*}
for every $x\in \Z^d$ and $r\geq 1$.
If $r \succeq \langle x \rangle^2$ and $\ell$ has the right parity then $p_\ell(0,x) \asymp r^{-d/2}$.
It follows that 
\begin{equation}
\label{eq:conditionalfirstmoment}
\E_{\mu,0}\left[\sum_{\ell=r}^{2r} B_\ell (x) \bigg| \sum_{\ell=r}^{2r} B_\ell (x) >0 \right] \succeq r^{(4-d)/2} 
\end{equation}
for every $x\in \Z^d$ and $r\geq \langle x \rangle^2$.

Suppose that $d \in \{1,2\}$. 
We deduce from \eqref{eq:conditionalfirstmoment}, \eqref{eq:low_d_2nd_moment} and the Paley-Zygmund inequality that there exists a constant $c>0$ such that if $r\geq \langle x \rangle^2$ then
\[
\P_{\mu,0}\left[L(x) \geq c  r^{(4-d)/2} \right] \geq \P_{\mu,0}\left[\sum_{\ell=r}^{2r} B_\ell (x) \geq c  r^{(4-d)/2} \right] \succeq \frac{r^{2-d}}{r^{3-d}}=\frac{1}{r},
\]
 and the desired lower bound follows by taking $r=\lceil (n/c)^{2/(4-d)}\rceil $.
Now suppose that $d=3$. 
Applying \eqref{eq:generalized_Paley_Zygmund} with $p=3$ we obtain that there exists a constant $c$ such that 
\[
\P_{\mu,0}\left[L(x) \geq c  r^{(4-d)/2} \right] \geq \P_{\mu,0}\left[\sum_{\ell=r}^{2r} B_\ell (x) \geq c  r^{(4-d)/2} \right] \succeq r^{-1},
\]
and we conclude as before.
\end{proof}

\section{High dimensions}

In this section we treat the case $d \geq 5$. 

\begin{prop}
\label{prop:high_d}
Let $d\geq 5$ and suppose that the offspring distribution $\mu$ is critical, non-trivial, and sub-exponential. Then
\[
\P_{\mu,0}(L(x)\geq n) =  \exp\left[-\Theta(n)\right] \langle x \rangle^{-d+2}
\]
for every $n\geq 1$ and $x\in \Z^d$.
\end{prop}

The lower bound is simple, and most of our work will go into proving the upper bound. 
By a standard computation, which we reproduce below, it  suffices to prove that there exists a constant $C=C(\mu,d)$ such that $\E_{\mu,0}[L(x)^k] \leq C^k k! \langle x \rangle^{-d+2}$ for every $k\geq 1$.
Thus, applying \cref{lem:moments_to_diagrams}, it suffices to prove the following two lemmas.
Recall that $c(u)$ is the number of offspring of a vertex $u$ in a skeleton.

\begin{lemma}[The skeleton partition function]
\label{lem:skeleton_partition}
If $\mu$ is critical and sub-exponential then there exist a constant $\kappa=\kappa(\mu)$ such that $\sum_{S \in \cS_k} \prod_{u\in S} b_{c(u)} \preceq \kappa^k k!$ for every $k\geq 1$.
\end{lemma}

\begin{lemma}[Contribution of a single skeleton]
\label{lem:high_d_diagrams}
Let $d\geq 5$. There  exists a constant $\lambda=\lambda(d)$ such that
\[\bD(0,x,\ldots,x;S) \leq \lambda^k \langle x \rangle^{-d+2}\] for every $k \geq 0$, $S \in \cS_k$ and $x\in \Z^d$.
\end{lemma}

We begin with \cref{lem:skeleton_partition}.

\begin{proof}[Proof of \cref{lem:skeleton_partition}]
Since $\mu$ is subexponential it satisfies a bound of the form $\mu(n) \leq C \lambda^n$ for some $C<\infty$ and $\lambda <1$. Thus, we have by a standard generating function calculation \cite[Eq.\ 1.31]{MR2172781} that
\[
b_k \leq C \sum_{n =k}^\infty \binom{n}{k}\lambda^n = \frac{C}{1-\lambda} \left(\frac{\lambda}{1-\lambda}\right)^k
\]
for every $k\geq 0$. 
Since $\sum_{u\in S} c(u) = |V(S)|-1$ for every skeleton $S$, it follows that
\[
\prod_{u\in S} b_{c(u)} \leq \left(\frac{C}{1-\lambda}\right)^{|V(S)|}\left(\frac{\lambda}{1-\lambda}\right)^{|V(S)|-1}
\]
for every skeleton $S$.

Let $\cS_{n,k} \subseteq \cS_k$ be the set of isomorphism classes of $k$-skeletons with exactly $n$ vertices,
and let $\cT_n$ denote the set of isomorphism classes of rooted plane trees with exactly $n$ vertices. It is well known \cite[Example 2.16]{MR2172781} that $|\cT_n|$ is given by the Catalan number
\begin{equation}
\label{eq:Catalan}
|\cT_n| = \frac{1}{n} \binom{2n-2}{n-1}
\leq 4^n. 
\end{equation}
For each rooted plane tree $T \in \cT_n$ there are at most $n^k$ isomorphism classes of $k$-skeletons with underlying rooted tree $T$, so that
$|\cS_{n,k}| \leq 4^n n^k$ for every $n\geq 1$ and $k\geq 0$. 
On the other hand, if $S \in \cS_k$ then every vertex of $V^\circ(S)$ has degree at least three, so that
\begin{equation}
\label{eq:vertex_count}
3|V^\circ(S)| + |\partial V(S)| \leq \sum_{u \in V} \deg(u) =  2|V|-2 = 2|V^\circ(S)| + 2 |\partial V(S)| -2
\end{equation} and hence that $|V(S)| \leq 2k$. 
Putting these observations together, we obtain that
\begin{align*}
\sum_{S \in \cS_k} \prod_{u\in S} b_{c(u)} &\leq \sum_{S \in \cS_k} \left(\frac{C}{1-\lambda}\right)^{|V(S)|}\left(\frac{\lambda}{1-\lambda}\right)^{|V(S)|-1} 
\\&= \sum_{n=1}^{2k} |\cS_{n,k}| \left(\frac{C}{1-\lambda}\right)^{n}\left(\frac{\lambda}{1-\lambda}\right)^{n-1}
\leq \sum_{n=1}^{2k} 4^n n^k \left(\frac{C}{1-\lambda}\right)^{n}\left(\frac{\lambda}{1-\lambda}\right)^{n-1},
\end{align*}
for every $k\geq 0$, from which the claim follows easily.
\end{proof}

\cref{lem:high_d_diagrams} will be proven using the recursive inequality \cref{lem:recursion} together with the following simple fact,
which is related to the fact that the simple random walk \emph{bubble diagram} converges when $d\geq 5$.

\begin{lemma} 
\label{lem:bubble}
Let $d\geq 5$. Then there exists a constant $C=C(d) \geq 1$ such that 
\[
C^{-1}\langle x \rangle^{-d+2} \leq \sum_{y \in \Z^d} \langle y \rangle^{-2d+4}\langle x-y\rangle^{-d+2} \leq C \langle x \rangle^{-d+2}
\]
for every $x \in \Z^d$.
\end{lemma}

\begin{proof}[Proof of \cref{lem:bubble}]
  The lower bound is trivial from the contribution of $y=0$.
  For the upper bound, consider the set $A = \{y \in \Z^d : d(x,y) \geq d(0,x)/2\}$.
  We will control the contribution to the sum from $A$ and $A^c$ separately. 
  If $y\in A$ then $\langle x-y \rangle \succeq \langle x \rangle$, so that
  \begin{equation}
    \label{eq:Abound}
    \sum_{y \in A} \langle y \rangle^{-2d+4}\langle x-y\rangle^{-d+2} \preceq \sum_{y \in A} \langle y \rangle^{-2d+4} \langle x \rangle^{-d+2} \preceq \langle x \rangle^{-d+2},
  \end{equation}
  where we used that $\sum_{y \in \Z^d}\langle y \rangle^{-2d+4}$ is finite when $d\geq 5$. On the other hand, if $y \in A^c$ then $d(0,x)/2 \leq d(0,y) \leq 3d(0,x)/2$ and we have that
  \[
    \sum_{y \in A^c} \langle y \rangle^{-2d+4} \langle x-y \rangle^{-d+2} \preceq \langle x \rangle^{-2d+4} \sum_{y \in  B} \langle x -y \rangle^{-d+2}.
  \]
Since there are $O(r^{d-1})$ points $y$ with $\langle x - y \rangle = r$ for each $r\geq 1$, we deduce that
\begin{equation}
\label{eq:Bbound}
\sum_{y \in B} \langle y \rangle^{-2d+4} \langle x -y \rangle^{-d+2} \preceq \langle x \rangle^{-2d+4} \sum_{r=1}^{3\langle x \rangle /2} r \preceq \langle x \rangle^{-2d+6} \preceq \langle x \rangle^{-d+2}
\end{equation}
where we used that $d \geq 4$ in the last inequality. Combining \eqref{eq:Abound} and \eqref{eq:Bbound} completes the proof.
\end{proof}

\begin{proof}[Proof of \cref{lem:high_d_diagrams}]
Let $C_1 \geq 1$ be such that $\tilde \bG(x,y) \leq C_1 \langle x-y \rangle^{-d+2}$ for every $x,y \in \Z^d$, let $C_2 \geq 1$ be the constant from \cref{lem:bubble}, and let $\lambda  = C_1^2 C_2[1 \vee \tilde \bG(0,0)^{-1}]$. We will prove by induction on $k$ that
\begin{equation}
\label{eq:high_d_induction}
M_k(x) \leq C_1 \lambda^{k-1} \langle x \rangle^{-d+2}
\end{equation}
for every $k\geq 1$. The base case $k=1$ is immediate, since $\cS_1'$ has only one element and this element $S$ has $\bD(0,x;S)=\tilde \bG(0,x) \leq C_1 \langle x \rangle^{-2}$. Now suppose that $k\geq 2$ and that the induction hypothesis \eqref{lem:high_d_diagrams} holds for all $1 \leq r \leq k-1$. Applying \cref{lem:recursion} and \cref{lem:bubble} we obtain that
\begin{align*}
M_k(x) &\leq \left[1\vee \tilde \bG(0,0)^{-1}\right] \max\left\{ C_1^3 \lambda^{k-r-1} \lambda^{r-1} \sum_{y \in \Z^d}  \langle y \rangle^{-2d+4} \langle x -y \rangle^{-d+2} : 1 \leq r \leq k-1\right\}
\\&\leq \left[1 \vee \tilde \bG(0,0)^{-1}\right] C_1^3 C_2 \lambda^{k-2}  \langle x \rangle^{-d+2} \leq C_1 \lambda^{k-1} \langle x \rangle^{-d+2}
\end{align*}
for every $x \in \Z^d$. This completes the induction.

The claim follows from \eqref{eq:high_d_induction} and \eqref{eq:max_noninjective}.
\end{proof}

\begin{proof}[Proof of \cref{prop:high_d}]
  We begin with the upper bound. Lemmas \ref{lem:moments_to_diagrams} , \ref{lem:skeleton_partition}, and \ref{lem:high_d_diagrams} imply that there exists a constant $\alpha$ such that 
  $\E_{\mu,0}[L(x)^k] \leq \alpha^k k! \langle x \rangle^{-d+2}$
  for every $k \geq 1$ and $x\in \Z^d$.
We deduce that
  \begin{multline}
    \E_{\mu,0}\left[e^{L(x)/2\alpha} \mathbbm{1}(L(x)>0)\right] \leq \frac{e^{1/2\alpha}}{e^{1/2\alpha}-1}\E_{\mu,0}\left[e^{L(x)/2\alpha} -1\right]
    \\
    = \frac{e^{1/2\alpha}}{e^{1/2\alpha}-1}\sum_{k \geq 1} \frac{1}{2^k \alpha^k k!} \E_{\mu,0}\left[L(x)^k\right] \leq \frac{e^{1/2\alpha}}{e^{1/2\alpha}-1} \langle x \rangle^{-d+2}
  \end{multline}
  for every $x\in \Z^d$, and hence by Markov's inequality that
  \[
    \P_{\mu,0}(L(x)\geq n) \leq \frac{e^{-(n-1)/2\alpha} \langle x \rangle^{-d+2}}{e^{1/2\alpha}-1}
  \]
  for every $x\in \Z^d$ and $n\geq 1$ as claimed.

\medskip

We finish with the lower bound. First suppose that $x=0$.
The probability $q$ that the initial particle has at least one grandchild is positive, and any grandchild has probability $1/(2d)$ of being back at the origin.
By the Markov property, the probability that there are at least $n$ visits to $0$ is at least $(q/2d)^n = e^{-\Theta(n)}$ for every $n\geq 1$.
If $x \neq 0$, then we claim that
\begin{align*}
  \P_{\mu,0}(L(x)\geq n)
  &\geq \P_{\mu,0}(L(x)>0)\P_{\mu,x}(L(x) \geq n) \\
  &= \P_{\mu,0}(L(x)>0)\P_{\mu,0}(L(0) \geq n) \succeq \langle x \rangle^{2-d} e^{-\Theta(n)}
\end{align*}
as required, where the final inequality follows from \eqref{eq:intro_hitting}. 
Indeed, for the first inequality, note that if we explore the genealogical tree $T$ in a breadth-first manner until $x$ is visited for the first time,  the part of the branching process that is descended from this first visit to $x$ has conditional law $\P_{\mu,x}$. This completes the proof.
\end{proof}

\section{The critical dimension}

In this section we deal with the case of the upper critical dimension $d=4$, which is the most technical. We rely on the machinery developed in the previous sections, in particular \cref{lem:skeleton_partition} and \cref{lem:recursion}.
The following is the $d=4$ case of \cref{thm:main_offdiagonal}.

\begin{prop}
\label{prop:4d_tail}
Let $d=4$ and suppose that the offspring distribution $\mu$ is critical, nontrivial, and subexponential. Then
\[
\P_{\mu,0}(L(x)\geq n) \asymp  \exp\bigg[ -\Theta \!\left(\min\left\{\sqrt{n}, \tfrac{n}{\log \langle x \rangle}\right\}\right)\bigg] \langle x \rangle^{-2} \log^{-1}\langle x \rangle
\]
for every $n\geq 1$ and $x\in \Z^d$.
\end{prop}

\begin{remark}
  \cref{prop:4d_tail} shows that in four dimensions, unlike in low dimensions, the easiest way for $L(0)$ to be large is \emph{not} for the genealogical tree to be ``large in a typical way''.
  Indeed, $L(0)$ is typically logarithmic in the size of the tree, so for $L(0)$ to be of order $n$ we would need the tree to survive to generation $e^{\Omega(n)}$. This occurs with probability $e^{-\Omega(n)}$, which is much smaller than the probability that $L(0) \geq n$.
\end{remark}

The proof of this proposition relies on the results of Zhu \cite{MR3962482,zhu2017critical} (i.e., the $d=4$ case of the hitting probability estimate \eqref{eq:intro_hitting}) in the case $x \neq 0$, but is self-contained in the case $x=0$. Indeed, the proposition will follow from Zhu's results together with the following proposition.

\begin{prop}
\label{prop:4d_moments}
Let $d=4$ and suppose that the offspring distribution $\mu$ is critical, non-trivial, and subexponential. Then
there exist positive constants $c$ and $C$ such that
\[ c^k k! [k+\log \langle x \rangle]^{k-1}\langle x \rangle^{-2} \leq \E_{\mu,0}[L(x)^k] \leq C^k k! [k+\log \langle x \rangle]^{k-1} \langle x \rangle^{-2}\]
for every  $x\in \Z^d$ and $k\geq 1$.
\end{prop}


We begin with the following lemma, which is the four-dimensional analogue of \cref{lem:bubble}.

\begin{lemma} 
\label{lem:4dbubble}
Let $d=4$. Then there exists a positive constant $C$ such that 
\[
  \sum_{y \in \Z^d} \langle x-y\rangle^{-2} \langle y \rangle^{-4}  [k+\log \langle y \rangle]^k \leq \frac{C\langle x \rangle^{-2}}{k+1} [k+1+ \log\langle x \rangle ]^{k+1}
\]
for every $x \in \Z^d$ and $k\geq 0$.
\end{lemma}


\begin{proof}[Proof of \cref{lem:4dbubble}]
  Partition $\Z^3$ into three sets $A,B,C$ according to the distance to $0$ and $x$:
  \begin{align*}
    A &= \{ y \in \Z^d : d(0,y) \leq 2 d(0,x) \text{ and } d(x,y) \geq d(0,x)/2 \}, \\
    B &= \{ y \in \Z^d : d(x,y) < d(0,x)/2\}, \\
    C &= \{y \in \Z^d : d(0,y) > 2 d(0,x)\}.
  \end{align*}
  We will control the contribution to the sum of each of these three sets separately. 
If $y\in A$ then $ \langle x \rangle/2 \leq \langle x -y \rangle \leq 3\langle x \rangle$, so that
\begin{align}
\sum_{y \in A} \langle x-y\rangle^{-2} \langle y \rangle^{-4} [k+\log \langle y \rangle]^k 
&\asymp \langle x \rangle^{-2} \sum_{y \in A}  \langle y \rangle^{-4} [k+\log \langle y \rangle]^k
 \asymp \langle x \rangle^{-2} \sum_{r=1}^{2 \langle x \rangle} r^{-1} [k+\log r]^k.  \nonumber
\end{align}
The sum on the right hand side can be bounded with a little calculus: We have the integral identity
\[\int_1^s t^{-1}(k+ \log t)^k \dif t = \frac{(k+\log s)^{k+1}}{k+1} - \frac{k^{k+1}}{k+1}\]
for every $s \geq 1$,  
and since the function $t^{-1}(k+\log t)^k$ is decreasing when $t\geq 1$ (as can be seen by computing the derivative to be $-t^{-2} (k+\log t)^{k-1} \log t$), we have that
\begin{multline*}\sum_{r=1}^{2\langle x \rangle} r^{-1}[k+\log r]^k  \leq k^k + \int_1^{2\langle x  \rangle} t^{-1}[k+\log t]^k \dif t =  \frac{[k+\log 2\langle x \rangle]^{k+1}}{k+1} + \frac{k^k}{k+1}\\ \leq \frac{2}{k+1} [k+1+\log\langle x \rangle]^{k+1}
\end{multline*}
and hence that
\begin{equation}
\sum_{y \in A} \langle x-y\rangle^{-2} \langle y \rangle^{-4} [k+\log \langle y \rangle]^k \preceq \frac{\langle x \rangle^{-2}}{k+1}[k+1+\log \langle x \rangle]^{k+1}
\end{equation}
as required.

It remains to upper bound the contributions from $B$ and $C$.
If $y \in B$ then $d(0,x)/2 \leq d(0,y) \leq 2d(0,x)$ and we have that 
\begin{align}
\label{eq:Bbound4d}
\sum_{y \in B}  \langle x -y \rangle^{-2} \langle y \rangle^{-4} [k+\log \langle y \rangle]^k 
&\asymp \langle x \rangle^{-4} [k+\log 2\langle x \rangle]^k \sum_{y \in  B} \langle x -y \rangle^{-2}
\nonumber\\
& \asymp \langle x \rangle^{-4} [k+1+\log \langle x \rangle]^k  \sum_{r=1}^{2\langle x \rangle } r 
\\
&\asymp \langle x \rangle^{-2}[k+1+\log \langle x \rangle]^k
\leq  \frac{\langle x \rangle^{-2}}{k+1}[k+1+\log \langle x \rangle]^{k+1}
\end{align}
as required.

Finally, if $y \in C$ then $d(x,y) \geq 
d(0,y)-d(0,x) \geq d(0,y)/2$ and $d(0,y) > d(0,x)$, so that
\begin{equation}
\label{eq:C_polar}
\sum_{y \in C} \langle x-y \rangle^{-2} \langle y \rangle^{-2} [k+\log \langle y \rangle]^k
\asymp \sum_{y \in C}  \langle y \rangle^{-6} [k+\log \langle y \rangle]^k
\end{equation}
Up to constants, there are $2^{4n}$ choices for $y$ with $2^n \le \langle y \rangle < 2^{n+1}$. For each such $y$ we have $\langle y \rangle^{-6} [k+\log \langle y \rangle]^k \asymp 2^{-6n} (k+n\log 2)^k$,  so the total contribution from all such $y$'s is (up to constants) $2^{-2n} (k+n\log 2)^k$.
Thus
\begin{equation}
\label{eq:first_term}
  \sum_{y \in C} \langle x-y \rangle^{-2} \langle y \rangle^{-2} [k+\log \langle y \rangle]^k
  \asymp \sum_{2^n > \langle x \rangle} 2^{-2n} (k+n\log 2)^k.
\end{equation}
The ratio of consecutive terms in this sum is
\[
  \frac{2^{-2(n+1)} (k+(n+1)\log 2)^k}{2^{-2n} (k+n\log 2)^k}
  \le \frac14 \left(1+\frac1{k+n\log 2}\right)^k \leq \frac{e}{4}.
\]
Since $e/4 < 1$, it follows that the sum on the right of \eqref{eq:first_term} is of the same order as its first term, and we deduce that
\begin{equation}
  \label{eq:Cbound}
  \sum_{y \in C} \langle x-y \rangle^{-2} \langle y \rangle^{-4} [k+\log \langle y \rangle]^k
  \asymp 
 \langle x \rangle^{-2} [k+ \log 2\langle x \rangle]^k
\preceq
  \frac{\langle x \rangle^{-2}}{k+1} [k+1+\log \langle x \rangle]^{k+1}.
\end{equation}
This is also of the required order, completing the proof.
\end{proof}

\begin{proof}[Proof of \cref{prop:4d_moments}]
  We begin with the upper bound.
  Let $C_1\geq 1$ be a constant such that \mbox{$\tilde \bG(0,x) \leq C_1 \langle x \rangle^{-2}$} for every $x\in \Z^4$, let $C_2 \geq 1$ be the constant from \cref{lem:4dbubble}, and let $\lambda = C_1^2 C_2 [1 \vee \tilde \bG(0,0)^{-1}]$. We prove by induction on $k$ that
\begin{equation}
M_k(x) \leq C_1 \lambda^{k-1} \langle x \rangle^{-2} [k-1+\log \langle x \rangle]^{k-1}
\label{eq:4dMbound}
\end{equation}
for every $k\geq 1$ and $x\in \Z^4$. The base case $k=1$ is trivial. For $k\geq 2$, we may apply \cref{lem:recursion} and the induction hypothesis to obtain that
\begin{multline*}
M_k(x) \leq [1 \vee \tilde \bG(0,0)^{-1}]C_1^3 \lambda^{k-2} 
\\
\cdot\max \Biggl\{
\sum_{y \in \Z^4} \langle y \rangle^{-4}\langle x -y \rangle^{-2}[k-r-1+\log \langle y \rangle]^{k-r-1}[r-1+\log \langle y \rangle]^{r-1}
: 1 \leq r \leq k-1 \Biggr\} 
\end{multline*}
and hence that
\begin{multline*}
M_k(x) \leq [1 \vee \tilde \bG(0,0)^{-1}]C_1^3 \lambda^{k-2}
\sum_{y \in \Z^4} \langle y \rangle^{-4}\langle x -y \rangle^{-2}[k-2+\log \langle y \rangle]^{k-2}
\\\leq C_1 \lambda^{k-1} \langle x \rangle^{-2}[k-1+\log \langle x \rangle]^{k-1}
\end{multline*}
as desired, where we applied \cref{lem:4dbubble} in the second line. As in the proof of \cref{prop:high_d}, it follows from \eqref{eq:4dMbound}, \cref{lem:moments_to_diagrams}, and \cref{lem:skeleton_partition} that there exists a constant $C_3$ such that
\begin{equation}
\label{eq:4d_main_upper}
\E_{\mu,0}[L(x)^k] \leq C_3^k k! [k-1+\log \langle x \rangle]^{k-1} \langle x \rangle^{-2}
\end{equation}
for every $x\in \Z^d$ and $k\geq 1$ as claimed.

\medskip

We now turn to the lower bound. We first prove the bound for $k$ of the form $2^\ell$ for some natural number $\ell \geq 1$. For each $\ell \geq 0$, let $k=2^\ell$ and let $T=T_\ell$ be the rooted plane tree with boundary in which the root has degree $1$, the descendants of the root's child form a complete binary tree of height $\ell$, and $\partial V(T_\ell)$ is equal to the set of leaves of $T$.
Let $\rho$ be the root of $T_\ell$, let $v_0$ be the child of the root, and for each vertex $v$ of $T_\ell$ other than $\rho$, let $\sigma(v)$ denote the parent of $v$ in $T$.
There are $k!$ ways to label the non-root leaves of $T$ with the labels $\{1,\ldots,k\}$, and each such labelling yields a distinct $k$-skeleton. Let $S=S_\ell$ be one such labelled $k$-skeleton. Applying \cref{lem:moments_to_diagrams}, we have by symmetry that
\begin{equation}
\E_{\mu,0}[L(x)^k] = \E_{\mu,x}[L(0)^k]\geq k! (b_2)^{k-1} \bD(x,0,\ldots,0;S).
\label{eq:4dpartition}
\end{equation}
(Recall that $b_2$ is the second descending moment of the offspring distribution, which is positive since $\mu$ is critical and nontrivial.)

Consider the set $\Phi$ of functions $\phi : V^\circ(T)  \to \{0,1,\ldots,k \vee \lceil \log_2 \langle x \rangle \rceil\}$ that are decreasing along each branch of the tree, i.e. such that
$\phi(\sigma(v))\geq \phi(v)$ for each $v\in V^\circ(T) \setminus \{v_0\}$. For each $\phi\in \Phi$, we define $\Lambda(\phi)$ to be the set of functions $f: V^\circ(T) \to \Z^4$ such that $2^{\phi(v)} \leq d(0, f(v)) < 2^{\phi(v)+1}$ for every $v\in V^\circ(T)$. Note that the sets $\Lambda(\phi)$ and $\Lambda(\psi)$ are disjoint whenever $\phi,\psi \in \Phi$ are distinct. Moreover, if $\phi \in \Phi$ and $f\in \Lambda(\phi)$ then
 $d(f(v),f(\sigma(v)) \leq 2 \max d(0,f(v)),d(0,f(\sigma(v))$ so that
 $\langle f(v)-f(\sigma(v)) \rangle \preceq \langle f(\sigma(v)) \rangle \asymp 2^{\phi(\sigma(v))}$ for every $v\in V^\circ \setminus \{v_0\}$. Similarly, we necessarily have that $\langle x - f(v_0) \rangle \preceq 2^{k \vee \log_2 \langle x \rangle} \leq 2^k \langle x \rangle$. Thus, we obtain from the definitions that there exists a positive constant $c_1$ such that
\begin{align*}
\bD(x,0,\ldots,0;S) 
& \geq c_1^k \langle x \rangle^{-2} \sum_{\phi \in \Phi} |\Lambda(\phi)|  \prod_{v\in V^\circ(T)}2^{-4\phi(v)}.
\end{align*}
Next observe that there exists a positive constant $c_2$ such that 
\[|\Lambda(\phi)| = \prod_{v\in V^\circ(T)} |\{y \in \Z^4 : 2^{\phi(v)} \leq d(0,y) < 2^{\phi(v)+1}\}| \geq c_2^k \prod_{v\in V^\circ(T)} 2^{4 \phi(v)},\]
so that there exists a positive constant $c_3$ such that
\[
\bD(x,0,\ldots,0;S) \geq 
 c_3^k \langle x \rangle^{-2} |\Phi|.
\]

It remains to estimate $|\Phi|$. 
Let $E_i$ be the set of edges of $T$ connecting vertices at distance $i$ from the root to the children of these vertices, so that $|E_i|=2^i$ for $0 \leq i \leq \ell$, and let $E' = \bigcup_{i=0}^{\ell-1} E_i$.  
Let $\Psi$ be the set of functions $\psi: E'  \to \{0,\ldots,k \vee \lceil \log_2 \langle x \rangle \rceil \}$ such that if $e\in E_i$ then $\psi(e) \leq 2^{i-\ell}(k \vee \lceil \log_2 \langle x \rangle\rceil )$. We clearly have that 
\[
|\Psi| = \prod_{m=1}^{\ell} \left\lfloor \frac{k \vee \lceil \log_2 \langle x \rangle \rceil  }{2^m}+1\right\rfloor^{2^{\ell-m}} \geq \prod_{m=1}^{\ell} \left[ \frac{k \vee \log_2 \langle x \rangle }{2^m}\right]^{2^{\ell-m}} = [k  \vee \log_2 \langle x \rangle]^{2^{\ell}-1} 2^{-\sum_{m=1}^{\ell} m2^{\ell-m}}
\]
We claim that there is an injection $\Psi \to \Phi$. 
Given $\psi \in \Psi$, let $\phi \in \Phi$ be defined recursively by $\phi(v_0)=k \vee \lceil \log_2 \langle x \rangle \rceil$ and $\phi(v) = \phi(\sigma(v))-\psi(\{v,\sigma(v)\})$ for every $v\in V^\circ(T) \setminus \{v_0\}$. The function $\phi$ is indeed an element of $\Phi$, since $\phi(v) \geq k \vee \log_2 \langle x \rangle - \sum_{i=1}^{\ell-1} 2^{i-\ell} (k \vee \log_2 \langle x \rangle) \geq 0$ for every $v\in V^\circ(T)$. Moreover, distinct elements of $\Psi$ clearly lead to distinct elements of $\Phi$ under this assignment, as claimed. We deduce that
\[
|\Phi| \geq |\Psi| \geq [k  \vee \log_2 \langle x \rangle]^{k-1} 2^{-\sum_{m=1}^{\ell-1} m2^{\ell-m}}
\geq c_4^k [k  \vee \log_2 \langle x \rangle]^{k-1}
\]
where $c_4=2^{-\sum_{m=1}^\infty m 2^{-m}}>0$. It follows that there exists a constant $c_5>0$ such that 
\begin{equation}
\bD(x,0,\ldots,0;S_\ell) \geq c_5^k [k + \log_2 \langle x \rangle]^{k-1} \langle x \rangle^{-2}
\label{eq:4dbinarydiagram}
\end{equation}
for every $k=2^\ell \geq 2$ and $x \in \Z^4$. Putting together \eqref{eq:4dpartition} and \eqref{eq:4dbinarydiagram}, we obtain that there exists a constant $c_6>0$ such that
\begin{equation}
\E_{\mu,0}[L(x)^k] \geq c_6^k k! \left[ k + \log \langle x \rangle\right]^{k-1} \langle x \rangle^{-2}
\label{eq:4d_dyadic}
\end{equation}
for every $x\in \Z^4$ and every $k=2^{\ell}$ for some $\ell\geq 1$. 

To get the lower bound for $k$ which is not a power of $2$, we interpolate using log-convexity.
By Cauchy-Schwarz, for any random variable $X\geq0$ and any $a\geq i\geq0$ we have
$\left(\E X^a\right)^2 \leq \left(\E X^{a-i}\right)\left(\E X^{a+i}\right)$, so that the moments $\E X^n$ are a log-convex sequence.
Since we have the claimed upper bound for every $k$ and the lower bound for powers of $2$, the lower bound follows for all $k$.
More precisely, let $a\in[k,2k]$ be a power of $2$, and let $b=2a-k$.
Log-convexity gives
\[
  \E_{\mu,0}\left[L(x)^k\right]
  \geq \frac{\E_{\mu,0}[L(x)^a]^2}{\E_{\mu,0}[L(x)^{b}]}.
\]
Applying \eqref{eq:4d_dyadic} to control the numerator and \eqref{eq:4d_main_upper} to control the denominator 
yields the lower bound for arbitrary $k$.
\end{proof}

\begin{proof}[Proof of \cref{prop:4d_tail}]
By Zhu's Theorem, it suffices to prove that
\[
\P_{\mu,0}(L(x) \geq n  \mid L(x)>0) = \exp\left[-\Theta\left(\min\left\{\sqrt{n}, \frac{n}{\log \langle x \rangle}\right\}\right)\right]
\]
for every $x\in \Z^4$ and $n \geq 1$. Moreover, Zhu's Theorem and \cref{prop:4d_moments} imply that there exist positive constants $c_1$ and $C_2$ such that
\[
c^k_1 e^{k \log k} [k \vee \log \langle x \rangle]^{k-1} \log \langle x \rangle \leq \E_{\mu,0}\left[L(x)^k \big| L(x)>0 \right] \leq C^k_1 e^{k\log k} [k \vee \log \langle x \rangle]^{k-1} \log \langle x \rangle
\]
for every $x\in \Z^4$ and $k\geq 1$, and hence that there exist positive constants $c_2$ and $C_2\geq 1$ such that
\begin{equation}
c_2^k e^{k \log k} [k \vee \log \langle x \rangle]^{k} \leq \E_{\mu,0}\left[L(x)^k \mid L(x)>0 \right] \leq C_2^k e^{k\log k} [k \vee \log \langle x \rangle]^{k}
\label{eq:4d_nearly_there}
\end{equation}
for every $x\in \Z^4$ and $k\geq 1$.

For the upper bound, we apply \eqref{eq:4d_nearly_there} and Stirling's approximation to obtain that there exists a constant $C_3$ such that
\begin{align*}\E_{\mu,0}\left[\exp\left(\frac{1}{2eC_2} \min\left\{ \frac{L(x)}{\log \langle x \rangle}, \sqrt{L(x)} \right\}\right) \mid L(x)>0\right] \hspace{-6cm}&\\
&= \sum_{k=0}^\infty \frac{(2eC_2)^{-k}}{k!} \E_{\mu,0} \left[\min\left\{\frac{L(x)}{\log \langle x \rangle}, \sqrt{L(x)} \right\}^k\mid L(x)>0\right]\\
&\leq 
\sum_{k=0}^\infty \frac{(2eC_2)^{-k}}{k!} \min\left\{\E_{\mu,0} \left[\frac{L(x)^k}{\log^k \langle x \rangle} \mid L(x)>0 \right], \E_{\mu,0}\left[L(x)^{k/2} \mid L(x)>0\right]\right\}
\\
&\leq 
\sum_{k=0}^{\lfloor \log \langle x \rangle\rfloor} \frac{(2e)^{-k}}{k!} e^{k\log k} + \sum_{k=1+\lfloor\log\langle x \rangle\rfloor}^\infty \frac{e^{-k}(2C_2)^{-k/2}}{k!} e^{k \log k} \leq C_3.
 \end{align*}
Thus, it follows by Markov's inequality that there exists a constant $C_3$ such that
\[
\P\left(L(x)\geq n \mid L(x)>0\right) \leq C_3\exp\left(-\frac{1}{2eC_2} \min\left\{ \frac{n}{\log \langle x \rangle}, \sqrt{n} \right\}\right)
\]
for every $x\in \Z^4$ and $n\geq 1$ as required.

For the lower bound, we apply the Paley-Zygmund inequality to obtain that there exists a positive constant $c_3$ such that
\begin{align*}
\P_{\mu,0}\left(L(x)^k \geq \frac{1}{2}c_2^k e^{k \log k}[k \vee \log \langle x \rangle]^k \right) 
&\geq 
\P_{\mu,0}\left(L(x)^k \geq \frac{1}{2}\E_{\mu,0}\left[L(x)^k \mid L(x)>0\right] \right) \\
&\geq \frac{1}{4} \E_{\mu,0}\left[L(x)^{2k}\right]^{-1}\E_{\mu,0}\left[L(x)^k\right]^2
\\
&\geq \frac{1}{4} 
\frac{c_2^{2k} e^{2 k \log k}[k \vee \log \langle x \rangle]^{2k}}{C_2^{2k} e^{2k \log 2k}[2 k \vee \log \langle x \rangle]^{2k}} \geq c_3^k 
\end{align*}
for every $k \geq 1$ and $x\in \Z^4$, and hence that there exists a positive constant $c_4$ such that
\[
\P_{\mu,0}\left(L(x) \geq c_4 k [k \vee \log \langle x \rangle] \right) \geq c_3^k
\]
for every $k\geq 1$ and $x\in \Z^4$. The claimed lower bound follows from this inequality by taking $k = \min\left\{ \left\lceil \sqrt{n/c_4} \right\rceil, \left\lceil n/(c_4 \log \langle x \rangle)\right\rceil \right\}$.
\end{proof}

\section*{Acknowledgments} 
The authors are grateful to the organizers of the Oberwolfach Workshop 
\emph{Strongly Correlated Interacting Processes}, where this work was initiated.
We thank Ed Perkins and Jean-Fran\c{c}ois Le Gall for helpful discussions on the literature.
OA is supported in part by an NSERC discovery grant.

\footnotesize{
  \bibliographystyle{abbrv}
  \bibliography{unimodularthesis}
}
\end{document}